\documentclass[12pt, letterpaper]{amsart}
\usepackage{amsfonts,amsthm,mathtools}
\usepackage[all,cmtip]{xy}
\usepackage{fullpage}
\usepackage{amsmath}
\usepackage{amssymb}
\usepackage{enumerate}
\usepackage{tikz}
\usepackage[backref]{hyperref}
\usepackage{cleveref}
\usepackage{verbatim}
\usepackage{cancel}
\usepackage{float}
\usepackage{longtable}
\usepackage{listings}
\usepackage{makecell}
\usepackage{comment}
\usepackage{tikz-cd}

\newtheorem{lem}{Lemma}[section]
\newtheorem{thm}[lem]{Theorem}
\newtheorem{proposition}[lem]{Proposition}
\newtheorem{prop}[lem]{Proposition}
\newtheorem{cor}[lem]{Corollary}
\newtheorem{conj}[lem]{Conjecture}
\newtheorem{ex}[lem]{Example}

\theoremstyle{definition}
\newtheorem{remark}[lem]{Remark}
\newtheorem{definition}[lem]{Definition}

\DeclareMathAlphabet{\curly}{U}{rsfs}{m}{n}

\newcommand{\Pic}{\operatorname{Pic}}
\newcommand{\gon}{\operatorname{gon}}

\newcommand{\Q}{\mathbb{Q}}
\newcommand{\C}{\mathbb{C}}

\newcommand{\Z}{\mathbb{Z}}
\newcommand{\N}{\mathbb{N}}
\newcommand{\F}{\mathbb{F}}

\newcommand{\SL}{\operatorname{SL}}

\newcommand{\PP}{{\mathbb P}}

\mathchardef\mhyphen="2D

\title{Tetragonal modular quotients $X_0^*(N)$}

\author{\sc Petar Orli\'c}
\address{Petar Orli\'c \\
University of Zagreb\\  
Bijeni\v{c}ka Cesta 30 \\
10000 Zagreb\\
Croatia}
\email{petar.orlic@math.hr}

\begin{document}
\begin{abstract}
    Let $N$ be a positive integer. For every $d\mid N$ such that $(d,N/d)=1$ there exists an Atkin-Lehner involution $w_d$ of the modular curve $X_0(N)$. The curve $X_0^*(N)$ is a quotient curve of $X_0(N)$ by $B(N)$, the group of all involutions $w_d$. In this paper we determine all quotient curves $X_0^*(N)$ whose $\C$-gonality is equal to $4$, as well as all curves $X_0^*(N)$ whose $\Q$-gonality is equal to $4$.
\end{abstract}


\maketitle

\section{Introduction}
Let $C$ be a smooth projective curve over a field $k$. Studying $k$-rational points on $C$ and points of small degree over $k$ is an interesting subject in arithmetic geometry. Morphisms $C\to Y$ defined over $k$ are a source of such points. This is one of the reasons to study morphisms from $C$.

The $k$-gonality of $C$, denoted by $\textup{gon}_k C$, is the least degree of a non-constant $k$-rational morphism $f:C\to\mathbb{P}^1$. For curves of genus $g\geq2$ there exists an upper bound for $\textup{gon}_k C$, linear in terms of the genus, see Proposition \ref{poonen}.

When $C$ is a modular curve, there also exists a linear lower bound for the $\C$-gonality. This was first proved by Zograf \cite{Zograf1987}. The constant was later improved by Abramovich \cite{abramovich} and by Kim and Sarnak in Appendix 2 to \cite{Kim2002}.

The gonality of the modular curve $X_0(N)$ and its quotients has been extensively studied over the years. Ogg \cite{Ogg74} determined all hyperelliptic curves $X_0(N)$, Bars \cite{Bars99} determined all bielliptic curves $X_0(N)$, Hasegawa and Shimura determined all trigonal curves $X_0(N)$ over $\C$ and $\Q$ \cite{HasegawaShimura_trig}, Jeon and Park \cite{JeonPark05} determined all tetragonal curves $X_0(N)$ over $\C$, and Najman and Orlić \cite{NajmanOrlic22} determined all curves $X_0(N)$ with $\Q$-gonality equal to $4$, $5$, or $6$, and also determined the $\Q$ and $\C$-gonality for many other curves $X_0(N)$.

Regarding the gonality of the quotients of the curve $X_0(N)$, Furumoto and Hasegawa \cite{FurumotoHasegawa1999} determined all hyperelliptic quotients, and Hasegawa and Shimura \cite{HasegawaShimura1999,HasegawaShimura2000,HasegawaShimura2006} determined all trigonal quotients over $\C$. Jeon \cite{JEON2018319} determined all bielliptic curves $X_0^+(N)$, Bars, Gonzalez, and Kamel \cite{BARS2020380} determined all bielliptic quotients of $X_0(N)$ for squarefree levels $N$, Bars and Gonzalez determined all bielliptic curves $X_0^*(N)$, and Bars, Kamel, and Schweizer \cite{bars22biellipticquotients} determined all bielliptic quotients of $X_0(N)$ for non-squarefree levels $N$, completing the classification of bielliptic quotients.

The next logical step is to determine all tetragonal quotients of $X_0(N)$. All tetragonal quotients $X_0^{+d}(N)=X_0(N)/\left<w_d\right>$ over $\C$ and $\Q$ were determined in \cite{Orlic2023,Orlic2024}. Here, we will do the same for the curves $X_0^*(N)$. We also determine all curves $X_0^*(N)$ of genus $4$ that are trigonal over $\Q$, thus completing the classification of all $\Q$-trigonal curves $X_0^*(N)$, which has already mostly been done by Hasegawa and Shimura \cite{HasegawaShimura2000}.

Our main results are the following theorems.

\begin{thm}\label{trigonalthm}
    The curve $X_0^*(N):=X_0(N)/B(N)$ ($B(N)$ is the group of all Atkin-Lehner involutions $w_d$ of $X_0(N)$) has genus $g=4$ and $\Q$-gonality equal to $3$ if and only if \[N\in\{137,148,172,201,202,214,219,242,254,260,262,305,319,345,395,434\}.\]
\end{thm}

\begin{thm}\label{tetragonalthm}
    The curve $X_0^*(N)$ has $\Q$-gonality equal to $4$ if and only if $N$ is one of the following $142$ values in the following table. Here $g$ is the genus of the curve $X_0^*(N)$.
\begin{center}
\begin{longtable}{|c|c|}

\hline
\addtocounter{table}{-1}
$N$ & $g$\\
    \hline

\makecell{$160,173,199,200,224,225,228,247,251,259,261,264,267,$\\$273,275, 280,300,306,308,311,321,322,334,335,341,342,$\\$350,354,355,366,370,374,385,399,426,483,546,570$} & $4$\\

\hline

\makecell{$157,192,208,212,216,218,226,235,237,250,263,278,316,$\\$339,364,371,376,377,382,391,393,396,402,406,407,410,$\\$413,414,418,435,438,440,442,444,465,494,495,551,$\\$572,574,595,630,645,663,714,770,798,910$} & $5$\\

\hline

\makecell{$163,197,211,223,244,265,269,272,274,297,301,332,336,$\\$340,359,369,375,447,470,564,598,620,627,670,780$} & $6$\\

\hline

\makecell{$193,232,241,257,268,281,288,296,320,326,360,368,372,$\\$423,456,460,504,558,658,660$} & $7$\\

\hline

$292,304,344,412,520,532,585,990$ & $8$\\

\hline

$328,528,560$ & $9$\\
 
\hline

\caption*{}
\end{longtable}
\end{center}

Furthermore, the curve $X_0^*(N)$ has $\C$-gonality equal to $4$ and $\Q$-gonality greater than $4$ if and only if $N$ is one of the following $17$ values in the following table.

\begin{center}
\begin{longtable}{|c|c|}

\hline
\addtocounter{table}{-1}
$N$ & $g$\\
    \hline

\hline

$378$ & $5$\\

\hline

$271,291,314,325,327,338,398,451,506,561,590,609,615,690,858$ & $6$\\

\hline

$243$ & $7$\\
 
\hline

\caption*{}
\end{longtable}
\end{center}
\end{thm}

We use a variety of methods to prove these results. Each section roughly corresponds to one method. The paper is organized as follows: \begin{itemize}
    \item In \Cref{preliminaries_section} we eliminate all but finitely many levels $N$ for which the curve $X_0^*(N)$ is not $\Q$-tetragonal.
    \item In \Cref{genus46section} we determine the $\Q$-gonality of genus $4$ curves and some genus $6$ curves. We deal with the remaining genus $6$ curves in \Cref{morphism_section}.
    \item In \Cref{Fpsection} we give lower bounds on the $\Q$-gonality via $\F_p$-gonality, either by counting points over $\F_{p^n}$ or computing the Riemann-Roch dimensions of degree $4$ effective $\F_p$-rational divisors. The level $N=378$ is more involved and requires a separate treatment via sieving in \Cref{subsection378}.
    \item In \Cref{betti_section} we determine the $\C$-gonality by computing Betti numbers $\beta_{i,j}$. We use them both to prove and disprove the existence of degree $4$ morphisms over $\C$ (and $\Q$) to $\PP^1$.
    \item In \Cref{morphism_section} we first use \texttt{Magma} to construct degree $4$ rational morphisms to $\PP^1$ by finding a degree $4$ rational divisor $D\geq0$ with Riemann-Roch dimension $2$. In the second part of the section we construct quotient maps from $X_0^*(N)$. We use them to construct degree $4$ rational maps to $\PP^1$, but also to give lower bound on the $\Q$-gonality for some levels $N$.
    \item In \Cref{thmproof_section} we combine all these results to obtain proofs of Theorems \ref{trigonalthm} and \ref{tetragonalthm}.
\end{itemize}

A lot of the results in this paper rely on \texttt{Magma} computations. The version of \texttt{Magma} used in the computations is V2.28-3. This is important because, for example, the function X0Nstar() was extensively used for obtaining models of curves $X_0^*(N)$. This function is not present in the version V2.27-1, for example. The codes that verify all computations in this paper can be found on
\begin{center}
    \url{https://github.com/orlic1/gonality_X0_star}.
\end{center}
All computations were performed on the Euler server at the Department of Mathematics, University of Zagreb with a Intel Xeon W-2133 CPU running at 3.60GHz and with
64 GB of RAM.

\section*{Acknowledgements}

Many thanks to Francesc Bars Cortina for permission to use \texttt{Magma} codes from his Github repository
\begin{center}
    \url{https://github.com/FrancescBars/Magma-functions-on-Quotient-Modular-Curves}.
\end{center} as templates and for providing useful references in the literature. I am also grateful to Filip Najman and Maarten Derickx for their helpful comments and suggestions. Finally, I thank the anonymous referee for their review which improved the paper.

This paper was written during my visit to the Universitat Autònoma de Barcelona and I am grateful to the Department of Mathematics of Universitat Autònoma de Barcelona for its support and hospitality.

The author was supported by the Croatian Science Foundation under the project no. IP-2022-10-5008, by the project ``Implementation of cutting-edge research and its application as part of the Scientific Center of Excellence for Quantum and Complex Systems, and Representations of Lie Algebras", PK.1.1.10.0004, European Union, European Regional Development Fund - Competitiveness and Cohesion Programme 2021–2027, and by the European Union – NextGenerationEU through the National Recovery and Resilience Plan 2021–2026, via an institutional grant of the University of Zagreb Faculty of Science, IK IA 1.1.3, Impact4Math.

\subsection*{AI Disclosure}

Anthropic's Claude Opus 5.5 was used during this work. After the author proposed reducing the case $N=378$ to determining $W^1_4(X)(\Q)$ via a Mordell–Weil sieve, Claude suggested using the rank-$0$ quotient of $J_0^*(378)$ to obtain the local condition in Proposition \ref{prop_378}, and helped write and debug the \texttt{Magma} code. It also checked the internal consistency of the paper. The author independently checked all AI-assisted material and takes full responsibility for the content of the paper.

\section{Preliminaries}\label{preliminaries_section}

We first give a very important result that we will use throughout the paper. It was briefly mentioned in the Introduction.

\begin{prop}[{\cite[Proposition A.1.]{Poonen2007}}]\label{poonen}
    Let $X$ be a curve of genus $g$ over a field $k$.
    \begin{enumerate}[(i)]
        \item If $L$ is a field extension of $k$, then $\textup{gon}_L(X)\leq \textup{gon}_k(X)$.
        \item If $k$ is algebraically closed and $L$ is a field extension of $k$, then $\textup{gon}_L(X)=\textup{gon}_k(X)$.
        \item If $g\geq2$, then $\textup{gon}_k(X)\leq 2g-2$.
        \item If $g\geq2$ and $X(k)\neq\emptyset$, then $\textup{gon}_k(X)\leq g$.
        \item If $k$ is algebraically closed, then $\textup{gon}_k(X)\leq\frac{g+3}{2}$.
        \item If $\pi:X\to Y$ is a dominant $k$-rational map, then $\textup{gon}_k(X)\leq \deg \pi\cdot\textup{gon}_k(Y)$.
        \item If $\pi:X\to Y$ is a dominant $k$-rational map, then $\textup{gon}_k(X)\geq\textup{gon}_k(Y)$.
    \end{enumerate}
\end{prop}

In order to determine all $\Q$ and $\C$-tetragonal curves $X_0^*(N)$, we need to provide an upper bound for possible levels $N$. We use Abramovich's lower bound on the $\C$-gonality of modular curves \cite[Theorem 0.1]{abramovich} with Kim and Sarnak's constant $\lambda=\frac{2^{15}}{325}$ \cite[Appendix 2]{Kim2002}. This result was more clearly stated in \cite{JeonKimPark06}.

\begin{thm}\cite[Theorem 1.2.]{JeonKimPark06}\label{kimsarnakbound}
    Let $X_\Gamma$ be the algebraic curve corresponding to a congruence subgroup $\Gamma\subseteq \SL_2(\Z)$ of index
    $$D_\Gamma=[\SL_2(\Z):\pm\Gamma].$$
    If $X_\Gamma$ is $d$-gonal, then $D_\Gamma\leq \frac{2^{15}}{325}d$.
\end{thm}

\begin{cor}
    The curve $X_0^*(N)$ is not $\C$-tetragonal for $N\geq12906$.
\end{cor}

\begin{proof}
    If the curve $X_0^*(N)$ is tetragonal. then we have a composition map of degree $4\cdot2^{\omega(N)}$ ($\omega(N)$ is the number of different prime divisors of $N$) \[X_0(N)\xrightarrow{2^{\omega_N}}X_0^*(N)\xrightarrow{4}\PP^1.\] Since the index of $\Gamma_0(N)$ in $\SL_2(\Z)$ is equal to $\psi(N)=N\prod_{q\mid N}(1+\frac{1}{q})\geq N$, we have the following inequality: \[N\leq\psi(N)\leq\frac{2^{15}}{325}\cdot4\cdot2^{\omega(N)}.\] We now have $2$ cases.
    \begin{itemize}
        \item $\omega(N)\geq 6$: In this case we have \[\psi(N)\geq\prod_{q\mid N}(q+1)\geq 3\cdot4\cdot6\cdot8\cdot12\cdot14\cdot18^{\omega(N)-6}>\frac{2^{15}}{325}\cdot4\cdot2^{\omega(N)}.\] The last inequality can be easily checked by hand. Thus in this case we have a contradiction.
        \item $\omega(N)\leq5$: In this case we have \[N\leq\psi(N)\leq\frac{2^{15}}{325}\cdot4\cdot2^5<12906.\]
    \end{itemize}
\end{proof}

We are now left with only finitely many levels $N$ to consider. We can further reduce the number of possible levels such that $X_0^*(N)$ is $\Q$-tetragonal with Ogg's inequality \cite[Theorem 3.1]{Ogg74}, stated in simpler form in \cite[Lemma 3.1]{HasegawaShimura_trig}.

\begin{lem}[Ogg]\label{lemmaogg}
    Let $p$ be a prime not dividing $N$. Then the number of $\F_{p^2}$ points on $X_0(N)$ is at least
    $$L_p(N):=\frac{p-1}{12}\psi(N)+2^{\omega(N)}.$$
    Here $\psi(N)=N\prod_{q\mid N}(1+\frac{1}{q})$ is the index of the congruence subgroup $\Gamma_0(N)$, and $\omega(N)$ is the number of different prime divisors of $N$.
\end{lem}
\begin{lem}\cite[Lemma 3.5]{NajmanOrlic22}\label{Fp2points}
    Let $C/\Q$ be a curve, $p$ a prime of good reduction for $C$, and $q$ a power of $p$. Suppose $\#C(\F_q)>d(q+1)$. Then $\textup{gon}_\Q(C)>d$.
\end{lem}

If $N$ is a prime power, then we have $X_0^+(N)=X_0^*(N)$. Since all tetragonal curves $X_0^+(N)$ were determined in \cite{Orlic2023}, we can skip the proof for these levels. From now on we will therefore suppose that $N$ is not a prime power.

We may also eliminate all levels $N$ such that $g(X_0^*(N))\leq3$ because such curves have $\Q$-gonality at most $3$. The curve $X_0^*(N)$ is of genus $4$ if and only if \begin{align*}
    N\in\{&148,160,172,200,201,202,214,219,224,225,228,242,247,254,259,260,261,\\
&262,264,267,273,275,280,300,305,306,308,319,321,322,334,335,341,\\
&342,345,350,354,355,366,370,374,385,395,399,426,434,483,546,570\}.
\end{align*}

After eliminating all hyperelliptic and trigonal levels $N$ with $g\geq5$ (these are $N\in\{253,279,302,323,351,555\}$, other levels are prime powers - \cite{Hasegawa1997, HasegawaShimura2000}), as well as levels $N$ that do not satisfy Lemmas \ref{lemmaogg} and \ref{Fp2points} (with the least $p\nmid N$), there are $553$ possible curves $X_0^*(N)$ of $\Q$-gonality $4$ with genus $g\geq 5$. The largest levels out of them are \[6930,6510,6090,5460,4830,4620,4290,3990,3570,3360,3150,2940,2730,2560,\ldots\]

Also, we can eliminate levels $N\in\{2520,2940,3150,3360,4620,4830,5460,6090,6510,6930\}$ with Abramovich's bound (Theorem \ref{kimsarnakbound}). We are left with $543$ levels $N$ we need to check. The largest of them are \[3990,3570,2730,2580,\ldots\] We will solve all of them in the following sections. A complete list of these levels can be found on Github.

\section{Curves of genus $4$ and $6$}\label{genus46section}

In this section we determine the $\Q$ and $\C$-gonality of all curves $X_0^*(N)$ of genus $4$ and some curves of genus $6$. This is done with \texttt{Magma} functions Genus4GonalMap() and Genus6GonalMap(). We deal with the rest of genus $6$ curves in \Cref{morphism_section} where we give an explicit map to $\PP^1$ instead of just proving its existence.

\begin{prop}\label{genus4prop}
    The curve $X_0^*(N)$ has $\Q$-gonality equal to $3$ for 
    \[N\in\{148,172,201,202,214,219,242,254,260,262,305,319,345,395,434\}.\]

    The curve $X_0^*(N)$ has $\C$-gonality equal to $3$ and $\Q$-gonality equal to $4$ for \begin{align*}
        N\in\{&160,200,224,225,228,247,259,261,264,267,273,275,280,300,306,308,321,\\
        &322,334,335,341,342,350,354,355,366,370,374,385,399,426,483,546,570\}.
    \end{align*}
\end{prop}

\begin{proof}
    All these curves are of genus $4$. We use the \texttt{Magma} function Genus4GonalMap() to obtain a degree $3$ morphism to $\PP^1$ (it exists due to Proposition \ref{poonen} (v)). If the curve is $\Q$-trigonal, this morphism will be defined over $\Q$. Otherwise, the function gives a morphism over a quadratic or a biquadratic field \cite{magma} (even though the trigonal map can always be defined over a number field of degree $\leq2$).
\end{proof}

\begin{prop}\label{genus6prop}
    The curve $X_0^*(N)$ has $\Q$-gonality equal to $4$ for the following $12$ levels \[N\in\{265,274,297,301,369,375,447,470,564,598,627,670\}.\] The curve $X_0^*(N)$ has $\C$-gonality equal to $4$ and $\Q$-gonality greater than $4$ for the following $14$ levels \[N\in\{291,314,325,327,338,398,451,506,561,590,609,615,690,858\}.\]
\end{prop}

\begin{proof}
    All these curves are of genus $6$. We use the \texttt{Magma} function Genus6GonalMap() to obtain a degree $4$ morphism to $\PP^1$ (it exists due to Proposition \ref{poonen} (v)). This morphism is defined over $\Q$ for all levels $N$ in the first set.

    For all levels $N$ in the second set the function Genus6GonalMap() tells us that $X_0^*(N)$ is neither bielliptic nor isomorphic to a smooth plane quintic. Hence it has only finitely many gonal maps up to equivalence \cite{HARRISON20133} and the function gives us the map defined over a minimal degree number field \cite{magma}. For all these levels $N$ that degree $4$ map is not defined over $\Q$, hence these curves are not $\Q$-tetragonal. However, they are $\C$-tetragonal due to Proposition \ref{poonen} (v) (we know that they are not trigonal from \cite{HasegawaShimura2000}).
\end{proof}


We will construct degree $4$ rational morphisms from genus $5$ curves in a different way, by finding degree $4$ effective divisors with Riemann-Roch dimension $2$. This will be discussed further in \Cref{morphism_section}.

\section{$\F_p$-gonality}\label{Fpsection}

If $C/\Q$ is a curve and $p$ a prime of good reduction for $C$, then we have an inequality \[\textup{gon}_{\F_p}(C)\leq\textup{gon}_\Q(C).\]
$\F_p$ gonality is a powerful tool for giving a lower bound on the $\Q$-gonality. The reason for that is that there are only finitely many $\F_{p^n}$ points on $C$ for any $n\in\N$. Hence we can compute Riemann-Roch spaces of all degree $d$ $\F_p$-rational divisors $D\geq0$.

This section solves the majority of levels listed in \Cref{preliminaries_section}, $290$ levels $N$.

\begin{prop}\label{Fp_gonality}
The $\F_p$-gonality of $X_0^*(N)$ is at least $5$ for the following $154$ values of $N$:

\begin{center}
\begin{longtable}{|c|c||c|c||c|c||c|c||c|c||c|c||c|c||c|c|}
\hline
$N$ & $p$ & $N$ & $p$ & $N$ & $p$ & $N$ & $p$ & $N$ & $p$ & $N$ & $p$ & $N$ & $p$ & $N$ & $p$\\
    \hline

$388$ & $3$ & $394$ & $3$ & $400$ & $3$ & $417$ & $5$ & $422$ & $3$ & $424$ & $3$ & $432$ & $5$ & $436$ & $3$\\
\hline
$452$ & $3$ & $453$ & $2$ & $458$ & $3$ & $464$ & $3$ & $466$ & $3$ & $482$ & $3$ & $485$ & $3$ & $486$ & $5$\\
\hline
$488$ & $3$ & $489$ & $2$ & $493$ & $2$ & $500$ & $3$ & $507$ & $2$ & $513$ & $2$ & $515$ & $2$ & $517$ & $2$\\
\hline
$533$ & $2$ & $535$ & $2$ & $536$ & $3$ & $543$ & $2$ & $553$ & $2$ & $565$ & $2$ & $575$ & $2$ & $579$ & $2$\\
\hline
$583$ & $2$ & $589$ & $2$ & $600$ & $7$ & $612$ & $5$ & $616$ & $3$ & $621$ & $2$ & $624$ & $5$ & $629$ & $2$\\
\hline
$666$ & $5$ & $672$ & $5$ & $678$ & $5$ & $680$ & $3$ & $684$ & $5$ & $696$ & $5$ & $697$ & $2$ & $702$ & $5$\\
\hline
$708$ & $5$ & $720$ & $7$ & $726$ & $5$ & $728$ & $5$ & $730$ & $3$ & $732$ & $5$ & $738$ & $5$ & $740$ & $3$\\
\hline
$741$ & $2$ & $742$ & $3$ & $744$ & $5$ & $748$ & $3$ & $750$ & $7$ & $754$ & $3$ & $756$ & $5$ & $760$ & $3$\\
\hline
$762$ & $5$ & $765$ & $2$ & $774$ & $5$ & $786$ & $5$ & $792$ & $5$ & $795$ & $2$ & $804$ & $5$ & $806$ & $3$\\
\hline
$810$ & $7$ & $812$ & $3$ & $816$ & $5$ & $819$ & $5$ & $822$ & $5$ & $826$ & $5$ & $828$ & $5$ & $834$ & $5$\\
\hline
$836$ & $3$ & $840$ & $11$ & $846$ & $5$ & $852$ & $3$ & $854$ & $3$ & $874$ & $3$ & $876$ & $5$ & $882$ & $5$\\
\hline
$884$ & $3$ & $888$ & $5$ & $890$ & $3$ & $894$ & $5$ & $903$ & $2$ & $906$ & $5$ & $912$ & $5$ & $918$ & $5$\\
\hline
$938$ & $3$ & $942$ & $5$ & $954$ & $5$ & $957$ & $2$ & $960$ & $7$ & $962$ & $5$ & $969$ & $2$ & $978$ & $5$\\
\hline
$984$ & $5$ & $986$ & $3$ & $1038$ & $5$ & $1050$ & $11$ & $1062$ & $5$ & $1080$ & $7$ & $1092$ & $5$ & $1116$ & $5$\\
\hline
$1170$ & $7$ & $1218$ & $5$ & $1230$ & $7$ & $1254$ & $5$ & $1260$ & $11$ & $1290$ & $7$ & $1302$ & $5$ & $1320$ & $7$\\
\hline
$1326$ & $5$ & $1330$ & $11$ & $1386$ & $5$ & $1410$ & $7$ & $1428$ & $5$ & $1430$ & $7$ & $1470$ & $11$ & $1482$ & $5$\\
\hline
$1518$ & $5$ & $1530$ & $7$ & $1554$ & $5$ & $1560$ & $7$ & $1590$ & $7$ & $1596$ & $5$ & $1638$ & $5$ & $1650$ & $7$\\
\hline
$1680$ & $11$ & $1710$ & $7$ & $1770$ & $7$ & $1806$ & $5$ & $1830$ & $7$ & $1860$ & $7$ & $1890$ & $11$ & $1914$ & $5$\\
\hline
$1938$ & $5$ & $1950$ & $7$ & $1980$ & $7$ & $2010$ & $7$ & $2070$ & $7$ & $2130$ & $7$ & $2310$ & $13$ & $2730$ & $11$\\
\hline
$3570$ & $11$ & $3990$ & $11$ & & & & & & & & & & & &\\
 
    \hline
\end{longtable}
\end{center}
\end{prop}

\begin{proof}
    Using \texttt{Magma}, we compute that there are no functions of degree $\leq4$ in $\F_p(X_0^*(N))$. We are able to reduce the number of degree $4$ divisors that need to be checked by noting the following: If there exists a function $f$ over $\F_p$ of a certain degree and if $\#X_0^*(N)(\F_p)>d(p+1)$, then there exists $c\in \F_p$ such that the function $g(x):=\frac{1}{f(x)-c}$ has the same degree and its polar divisor contains at least $d$ $\F_p$-rational points, see \cite[Lemma 3.1]{NajmanOrlic22}.

    Therefore, if $\#X_0^*(N)(\F_p)>p+1$, it is enough to check divisors of the form $1+1+1+1$ and $1+1+2$, and if $\#X_0^*(N)(\F_p)>2(p+1)$, it is enough to just check divisors of the form $1+1+1+1$.
\end{proof}

The computations for some higher levels $N$ took several days to finish. For example, it took $2.5$ days for $N=3990$ and $7$ days for $N=3570$. The reason for that is that these curves are of high genus ($g(X_0^*(3990))=23, g(X_0^*(3570))=19$) and there are many $\F_{11}$-rational divisors of degree $4$ that must be checked. 

It also takes some time to find the model of $X_0^*(N)$ for higher levels $N$. It took $4$ hours to find the model of $X_0^*(3990)$ with the function X0Nstar(). Other methods were unable to even find a model in a reasonable time.

We can also get a lower bound on the $\Q$-gonality with Lemma \ref{Fp2points} by counting $\F_{p^n}$ points on the curve without actually computing the $\F_p$-gonality.

\begin{prop}\label{prop_counting_Fpn_points}
    The $\F_p$-gonality of $X_0^*(N)$ is at least $5$ for the following $135$ values of $N$:

\begin{center}
\begin{longtable}{|c|c|c||c|c|c||c|c|c|}
\hline
$N$ & $q$ & $\#X_0^*(N)(\F_q)$ & $N$ & $q$ & $\#X_0^*(N)(\F_q)$ & $N$ & $q$ & $\#X_0^*(N)(\F_q)$ \\
    \hline
$448$ & $9$ & $42$ & $472$ & $9$ & $42$ & $477$ & $4$ & $21$\\
\hline
$484$ & $9$ & $43$ & $496$ & $9$ & $44$ & $508$ & $9$ & $46$\\
\hline
$514$ & $9$ & $47$ & $538$ & $9$ & $41$ & $544$ & $9$ & $46$\\
\hline
$548$ & $9$ & $42$ & $549$ & $4$ & $21$ & $554$ & $3$ & $18$\\
\hline
$556$ & $9$ & $51$ & $562$ & $9$ & $47$ & $566$ & $9$ & $49$\\
\hline
$567$ & $4$ & $21$ & $568$ & $9$ & $44$ & $576$ & $25$ & $120$\\
\hline
$578$ & $9$ & $45$ & $584$ & $9$ & $42$ & $586$ & $9$ & $47$\\
\hline
$592$ & $9$ & $44$ & $597$ & $4$ & $23$ & $603$ & $4$ & $22$\\
\hline
$604$ & $9$ & $54$ & $605$ & $4$ & $22$ & $614$ & $9$ & $50$\\
\hline
$633$ & $4$ & $25$ & $635$ & $4$ & $22$ & $637$ & $4$ & $21$\\
\hline
$639$ & $4$ & $22$ & $649$ & $4$ & $21$ & $657$ & $4$ & $23$\\
\hline
$667$ & $9$ & $43$ & $669$ & $4$ & $22$ & $679$ & $9$ & $48$\\
\hline
$681$ & $4$ & $28$ & $685$ & $4$ & $21$ & $700$ & $9$ & $44$\\
\hline
$703$ & $9$ & $43$ & $707$ & $4$ & $23$ & $713$ & $4$ & $22$\\
\hline
$721$ & $4$ & $21$ & $725$ & $4$ & $21$ & $731$ & $4$ & $21$\\
\hline
$737$ & $4$ & $22$ & $745$ & $4$ & $26$ & $749$ & $4$ & $26$\\
\hline
$755$ & $4$ & $24$ & $763$ & $4$ & $24$ & $767$ & $4$ & $21$\\
\hline
$779$ & $4$ & $23$ & $781$ & $4$ & $25$ & $791$ & $4$ & $23$\\
\hline
$793$ & $4$ & $25$ & $799$ & $4$ & $21$ & $803$ & $4$ & $25$\\
\hline
$817$ & $4$ & $22$ & $820$ & $9$ & $42$ & $825$ & $4$ & $23$\\
\hline
$830$ & $9$ & $45$ & $850$ & $9$ & $41$ & $851$ & $4$ & $24$\\
\hline
$868$ & $9$ & $42$ & $880$ & $9$ & $45$ & $885$ & $4$ & $21$\\
\hline
$902$ & $9$ & $42$ & $915$ & $4$ & $22$ & $920$ & $9$ & $47$\\
\hline
$936$ & $25$ & $106$ & $940$ & $9$ & $47$ & $945$ & $4$ & $25$\\
\hline
$946$ & $9$ & $43$ & $950$ & $9$ & $47$ & $952$ & $9$ & $42$\\
\hline
$970$ & $9$ & $42$ & $975$ & $4$ & $23$ & $988$ & $9$ & $44$\\
\hline
$994$ & $9$ & $44$ & $1002$ & $25$ & $108$ & $1005$ & $4$ & $21$\\
\hline
$1008$ & $25$ & $108$ & $1010$ & $9$ & $48$ & $1014$ & $25$ & $114$\\
\hline
 $1015$ & $9$ & $42$ & $1022$ & $9$ & $48$ & $1023$ & $4$ & $21$\\
\hline
$1026$ & $25$ & $110$ & $1030$ & $9$ & $47$ & $1034$ & $9$ & $43$\\
\hline
$1035$ & $4$ & $22$ & $1054$ & $9$ & $42$ & $1056$ & $25$ & $112$\\
\hline
$1065$ & $4$ & $22$ & $1066$ & $9$ & $42$ & $1071$ & $4$ & $22$\\
\hline
$1074$ & $25$ & $105$ & $1085$ & $4$ & $21$ & $1086$ & $25$ & $112$\\
\hline
$1095$ & $4$ & $23$ & $1098$ & $25$ & $120$ & $1102$ & $9$ & $41$\\
\hline
$1104$ & $25$ & $112$ & $1105$ & $4$ & $22$ & $1113$ & $4$ & $24$\\
\hline
$1118$ & $9$ & $51$ & $1128$ & $25$ & $122$ & $1131$ & $4$ & $23$\\
\hline
$1146$ & $25$ & $112$ & $1158$ & $25$ & $111$ & $1173$ & $4$ & $24$\\
\hline
$1182$ & $25$ & $119$ & $1194$ & $25$ & $124$ & $1200$ & $49$ & $213$\\
\hline
$1206$ & $25$ & $116$ & $1209$ & $4$ & $23$ & $1221$ & $4$ & $23$\\
\hline
$1235$ & $4$ & $21$ & $1265$ & $4$ & $22$ & $1295$ & $4$ & $23$\\
\hline
$1309$ & $4$ & $25$ & $1540$ & $9$ & $43$ & $1610$ & $9$ & $45$\\
\hline
$1716$ & $25$ & $110$ & $1722$ & $25$ & $107$ & $1785$ & $4$ & $21$\\
\hline
$1794$ & $25$ & $105$ & $1974$ & $25$ & $115$ & $2040$ & $49$ & $230$\\
\hline
$2046$ & $25$ & $111$ & $2190$ & $49$ & $204$ & $2280$ & $49$ & $236$\\
\hline
$2370$ & $49$ & $206$ & $2490$ & $49$ & $222$ & $4290$ & $49$ & $221$\\
    \hline
\end{longtable}
\end{center}
\end{prop}

\begin{proof}
    Using \texttt{Magma}, we calculate the number of $\F_q$ points on $X_0^*(N)$. It is now easy to check that $\#X_0^*(N)(\F_q)>4(q+1)$ in all these cases.

    For the majority of these levels we used the function FpnpointsforQuotientcurveX0NWN(), available on \begin{center}
        \url{https://github.com/FrancescBars/Magma-functions-on-Quotient-Modular-Curves/blob/main/funcions.m}.
    \end{center} It computes the newforms $f$ such that the corresponding abelian variety $A_f$ is (up to isogeny) in the decomposition of the Jacobian $J_0^*(N)$ and, after that, the characteristic polynomial \[P(x)=\prod_f (x^2-a_p(f)x+p)=\prod_1^{2g}(x-\alpha_i).\] Here $g$ is the genus of $X_0^*(N)$. The number of $\F_{p^n}$-points on $X_0^*(N)$ is then equal to $p^n+1-\sum_1^{2g}\alpha_i^n$. For some non-squarefree levels $N$ this function fails to give the correct decomposition and returns a higher number of points than the correct one. In these cases we found the number of $\F_q$-points via the model of the curve. More details regarding the code are on Github.

    Most levels were computed very fast, with some higher levels taking $5-10$ minutes and $N=4290$ taking $30$ minutes.
\end{proof}

\subsection{$N=378$}\label{subsection378}

One can check that the curve $X_0^*(378)$ has degree $4$ rational morphisms over $\F_p$ to $\PP^1$ for all primes $p\leq79$ of good reduction. We use a more involved argument for this curve - sieving on the Jacobian $J_0^*(378)$.

\begin{definition}
    For a curve $X$ and divisor $D$ of degree $d$, $g_d^r$ is a subspace $V$ of the Riemann-Roch space $L(D)$ such that $\dim V=r+1$.
\end{definition}

Therefore, we want to determine whether $X_0^*(378)$ has a $\Q$-rational $g_4^1$.

\begin{proposition}\label{prop_378}
    The curve $X_0^*(378)$ has $\C$-gonality equal to $4$ and $\Q$-gonality equal to $5$.
\end{proposition}

\begin{proof}
    This curve is of genus $5$ and is neither hyperelliptic nor trigonal over $\C$ by \cite{Hasegawa1997, HasegawaShimura2000}. Moreover, it has a rational cusp, therefore $\gon_\C\leq\frac{g+3}{2}, \ \gon_\Q\leq g$ by \cite[Proposition A.1]{Poonen2007}. This proves that the $\C$-gonality is equal to $4$ and it remains to prove that the $\Q$-gonality is greater than $4$.
    
    Petri's model of this genus $5$ curve is the intersection of $3$ quadrics.
We can check this with \texttt{Magma} code \begin{align*}
    &\textup{X:=X0Nstar(378)};\\
    &\textup{CanonicalImage(X,CanonicalMap(X));}
\end{align*}
(just using the function X0Nstar() still yields a canonical model, but with cubics, not $3$ quadrics).

Let $X:=X_0^*(378)$ and $J:=J_0^*(378)$. Suppose that there exists a morphism $f:X\to\PP^1$ of degree $4$ defined over $\Q$. Then a rational divisor $D=f^*(\infty)$ gives a rational class $[D]\in W_4^1(X)(\Q)$. If $K_X$ is a canonical divisor, then $\deg K_X=2\cdot5-2=8$. By defining $y:=2D-K_X$ we get $[y]\in J(\Q)$.

We can easily check on LMFDB \cite{lmfdb} by looking at the Atkin-Lehner signs and newform lifts (\cite[Lemma 1]{Hasegawa1997}) that \[J\sim_\Q E_{21.2.a.a}\times E_{63.2.a.a}\times E_{126.2.a.a}\times E_{189.2.a.a}\times E_{378.2.a.c}.\] Out of these $5$ elliptic curves, $E_{189},E_{378}$ are of rank $1$, and $E_{21},E_{63},E_{126}$ are of rank $0$. Let $B\subset J$ be the abelian subvariety, stable under the Hecke operators, that corresponds to the forms of levels $189$ and $378$. Then \[A:=J/B\sim E_{21}\times E_{63}\times E_{126}\] has rank $0$, hence $A(\Q)$ is finite. It is known that for odd primes $q$ of good reduction, $A(\Q)=A(\Q)_\textup{tors}$ injects into $A(\F_q)$. Therefore, $\# A(\Q)\mid \# A(\F_q)$. Since the characteristic polynomial of Frobenius is an isogeny invariant, isogenous abelian varieties over a finite field have the same number of points, and one can check on LMFDB that \[\# A(\F_5)=\# E_{21}(\F_5)\cdot \# E_{63}(\F_5)\cdot \# E_{126}(\F_5)=8\cdot 4\cdot 4=128.\]
This means that $\# A(\Q)\mid 128$. From the definition we have that $B(\Q)=\ker(J(\Q)\to A(\Q))$, thus $128\cdot y\in B(\Q)$. 

We know that $X$ has good reduction at $q=29$ and it can easily be checked that \[\# B(\F_{29})=\# E_{189}(\F_{29})\cdot \# E_{378}(\F_{29})=34\cdot 34=2^2\cdot17^2.\] Therefore, $\tilde{y}$ and $\tilde{D}$ (the reductions mod $29$) are defined over $\F_{29}$ and  \begin{equation}\label{eq_jacobian}
    128\cdot \# B(\F_{29})\cdot \tilde{y}=2^9\cdot 17^2\cdot \tilde{y}=0\in J(\F_{29})\cong \Pic^0(X_{\F_{29}}).
\end{equation}

In other words, we have a commutative diagram \[\begin{tikzcd}[column sep=large]
W^1_4(X)(\mathbb{Q}) \arrow[r, "\varphi"] \arrow[d, "\mathrm{red}"]
  & J(\mathbb{Q}) \arrow[r, "\cdot 128"] \arrow[d, "\mathrm{red}"]
  & B(\mathbb{Q}) \arrow[d, "\mathrm{red}"] \\
W^1_4(X_{29})(\mathbb{F}_{29}) \arrow[r, "\varphi"]
  & J(\mathbb{F}_{29}) \arrow[r, "\cdot 128"]
  & J(\mathbb{F}_{29}) \supseteq B(\mathbb{F}_{29})
\end{tikzcd}\]

where $\varphi(D)=2D-K_X$. Let us define $M:=2^9\cdot 17^2$. Then if $W_4^1(X)(\Q)\neq\emptyset$, we would have \[\varphi(W^1_4(X_{29})(\mathbb{F}_{29}))\cap J(\F_{29})[M]\neq\emptyset.\]

By assumption we have that $\ell(\tilde{D})\geq\ell(D)\geq 2$, and $\gon_{\F_{29}}X_{\F_{29}}\geq4$ ($X_{\F_{29}}$ is a smooth complete intersection of three quadrics, so it is not trigonal over $\F_{29}$), implying that $\tilde{D}$ is an $\F_{29}$-rational $g_4^1$ and has no base points (a base point would leave $g_3^1$). Moreover, Clifford's theorem gives us $\ell(\tilde{D})\leq 1+\frac{\deg \tilde{D}}{2}=3$. The equality cannot occur since the curve is non-hyperelliptic and $\tilde{D}\nsim 0,\tilde{K}_X$ (they have different degrees). Therefore, $\ell(\tilde{D})=2$.

For a general curve of genus $5$ ($\gon_\C=4$) there are infinitely many equivalence classes of gonal maps which are parametrized by a plane quintic curve $F$ \cite{HARRISON20133, magma}. This plane quintic is in the output of the \texttt{Magma} function Genus5GonalMap() (in our case this function yields a degree $4$ map to $\PP^1$ defined over the quadratic field $\Q(\sqrt{-7})$, although this is not necessarily the smallest possible field so it does not prove the desired result).

More precisely, let $Q_1,Q_2,Q_3$ be the quadrics of the canonical model of $X$. The $g^1_4$'s on $X_{\F_{29}}$ with $\ell=2$ correspond to the points $\lambda$ of the plane quintic \[F: \det(\lambda_1M_1+\lambda_2M_2+\lambda_3M_3)=0,\] where $M_i$ is the matrix of $Q_i$, one for each ruling of the singular quadric $Q_\lambda=\lambda_1Q_1+\lambda_2Q_2+\lambda_3Q_3$ \cite{HARRISON20133}. A $g^1_4$ is defined over $\F_{29}$ exactly when $\lambda\in F(\F_{29})$ and the corresponding ruling of $Q_\lambda$ is defined over $\F_{29}$. The function Genus5GonalMap() returns $F$ and the function $f$ that sends $\lambda$ to the corresponding gonal map, defined over $\F_{29}$ or $\F_{29^2}$. The code finds that $\# F(\F_{29})=36$, and for $21$ of them the map is defined over $\F_{29}$.

As a small comment, when the rulings are rational, $f(\lambda)$ gives one of the two $g_4^1$'s, $[\tilde{D}]$ or $[\tilde{K}_X-\tilde{D}]$. Replacing $\tilde{D}$ by $\tilde{K}_X-\tilde{D}$ changes $\tilde{y}$ to $-\tilde{y}$. Condition \eqref{eq_jacobian} remains unchanged, so it does not matter which ruling $f$ picks.

For these $21$ candidates $\tilde{D}$, the code computes $\tilde{y}\in J(\F_{29})$ using \texttt{Magma} function ClassGroup() and checks that none of them satisfy Condition \eqref{eq_jacobian}. This means that \[\varphi(W^1_4(X_{29})(\mathbb{F}_{29}))\cap J(\F_{29})[M]=\emptyset\] and we get a contradiction. Therefore, $\gon_\Q(X_0^*(378))=5$.




\end{proof}

\begin{remark}
    The same procedure can also be done for other good primes, but $q=29$ is the first prime where the intersection is empty. The same test leaves no candidates for $q=41,73,89$ as well.
\end{remark}

\section{Betti numbers}\label{betti_section}

Graded Betti numbers $\beta_{i,j}$ are helpful when determining the $\C$-gonality of a curve. We will use the indexation of Betti numbers as in \cite[Section 1.]{JeonPark05}. The results we mention can be found there and in \cite[Section 3.1.]{NajmanOrlic22}. In this section we solve the majority of remaining levels, $135$ levels $N$.

Green's conjecture relates graded Betti numbers $\beta_{i,j}$ with the existence of $g_d^r$.

\begin{conj}[Green, \cite{Green84}]
    Let $X$ be a curve of genus $g$. Then $\beta_{p,2}\neq 0$ if and only if there exists a divisor $D$ on $X$ of degree $d$ such that a subspace $g_d^r$ of $L(D)$ satisfies $d\leq g-1$, $r=\ell(D)-1\geq1$, and $d-2r\leq p$.
\end{conj}

The "if" part of this conjecture has been proven in the same paper.

\begin{thm}[Green and Lazarsfeld, Appendix to \cite{Green84}]\label{thmGreenLazarsfeld}
    Let $X$ be a curve of genus $g$. If $\beta_{p,2}=0$, then there does not exist a divisor $D$ on $X$ of degree $d$ such that a subspace $g_d^r$ of $L(D)$ satisfies $d\leq g-1$, $r\geq1$, and $d-2r\leq p$.
\end{thm}

\begin{cor}[{\cite[Corollary 3.20]{Orlic2023}}]\label{betti0cor}
    Let $X$ be a curve of genus $g\geq5$ with $\beta_{2,2}=0$. Suppose that $X$ is neither hyperelliptic nor trigonal. Then $\textup{gon}_\C(X)\geq5$.
\end{cor}

We can say something more about the curve from the value of the Betti number $\beta_{2,2}$.

\begin{thm}[{\cite[Theorem 4.1, 4.4]{Schreyer91}}]\label{schreyer}
    Let $C\subseteq\PP^{g-1}$ be a reduced irreducible canonical curve over $\C$ of genus $g\geq7$. Then $\beta_{2,2}$ has one of the values in the following table.

    \begin{center}
        \begin{tabular}{|c|c|c|c|c|}
            \hline
             $\beta_{2,2}$ & $(g-4)(g-2)$ & $\displaystyle\binom{g-2}{2}-1$ & $g-4$ & $0$\\
             \hline
             \textup{linear series} & $\exists g_3^1$ & $\exists g_6^2\textup{ or } g_8^3$ & $\exists \textup{ a single } g_4^1$ & $\textup{no } g_4^1$\\
             \hline
        \end{tabular}
    \end{center}
\end{thm}

Notice that the indexation of Betti numbers is different here and in \cite{Schreyer91}, where they are indexed as $\beta_{2,4}$.

The existence of a single $g_4^1$ is useful as the following result tells us that in that case we have a rational morphism of degree $\leq4$ to $\PP^1$.

\begin{thm}[{\cite[Theorem 5]{RoeXarles2014}}]
    Let $C/k$ be a smooth projective curve. Suppose that for a fixed $r$ and $d$ there is only one $g_d^r$, giving a morphism $f:C_{k^{\textup{sep}}}\to \PP_{k^{\textup{sep}}}^r$. Then there exists a Brauer-Severi variety $\mathcal{P}$ defined over $k$ together with a $k$-morphism $g:C\to\mathcal{P}$ such that $g\oplus_k k^{\textup{sep}}:C_{k^{\textup{sep}}}\to\mathcal{P}_{k^{\textup{sep}}}\cong\PP_{k^{\textup{sep}}}^r$ is equal to $f$.
\end{thm}

\begin{cor}\label{unique_g14_cor}
    Let $C/\Q$ be smooth projective curve such that there is only one $g_d^1$. Then there exists a conic $\mathcal{P}$ defined over $\Q$ and a degree $d$ rational map $g:C\to\mathcal{P}$. If additionally $C$ has at least one rational point, then $\mathcal{P}$ is $\Q$-isomorphic to $\PP^1$.
\end{cor}

\begin{proof}
    In dimension $1$ Brauer-Severi varieties are conics. The degree of $f$ (and of $g$) cannot be $\leq d-1$ since $g_d^1$ is unique. Therefore, it is of degree $d$. If $C(\Q)\neq\emptyset$, then $\mathcal{P}(\Q)\neq\emptyset$ and it is isomorphic to $\PP^1$ over $\Q$.
\end{proof}

\begin{prop}\label{betti0prop}
    The $\C$-gonality of $X_0^*(N)$ is at least $5$ for the following $127$ values of $N$. Here $g$ is the genus of $X_0^*(N)$.

    \begin{center}
\begin{longtable}{|c|c|}

\hline
\addtocounter{table}{-1}
$N$ & $g$\\
    \hline

\makecell{$298,309,324,358,363,365,411,425,437,445,446,450,474,478,490,492,498,$\\$501,518,519,525,527,530,550,582,623,636,638,646,671,710,861,870,897,924$} & $7$\\

\hline

\makecell{$333,346,356,362,381,403,405,408,415,427,468,480,505,511,534,540,545,552,$\\$559,573,580,581,606,651,654,665,682,715,759,782,805,814,930,966,$\\$1001,1020,1155,1190$} & $8$\\

\hline

\makecell{$352,384,386,387,392,404,428,441,454,459,469,471,473,475,481,497,502,516,$\\$522,524,526,531,537,539,542,588,591,594,602,610,611,618,622,642,644,650,655,$\\$689,693,695,705,735,777,790,855,860,935,987,1045,1110,1122,1140,1365$} & $9$\\

\hline

$416$ & $10$\\
 
\hline

\caption*{}
\end{longtable}
\end{center}
\end{prop}

\begin{proof}
    For all these levels we compute that $\beta_{2,2}=0$. We know from \cite{HasegawaShimura2000} that these curves have $\C$-gonality at least $4$. Thus by Corollary \ref{betti0cor} they are not tetragonal over $\C$.

    The computations for $g=7$ levels took only several seconds, for $g=8$ up to several minutes, and for $g=9$ they could take around $1$ hour. The computation for the genus $10$ level $N=416$ took around $4$ hours. 
    
    In the code for this paper Betti numbers are computed with the \texttt{Magma} function BettiNumber(\_,2,4). The indexation of $\beta_{i,j}$ in that function is different from that in this paper. It is the same indexation as in \cite{Schreyer1986} and \cite{Schreyer91}, and the reader can consult \cite[Table 1]{Schreyer1986} to check this result.

    One of the requirements of Theorem \ref{schreyer} was that the curve must be canonical. The function X0Nstar() gives a canonical model, although not Petri's model. The provided model consists of cubics instead of quadrics, as mentioned in \Cref{genus46section}.
\end{proof}

\begin{prop}\label{betti_positive_prop}
    The curve $X_0^*(N)$ has $\Q$ and $\C$-gonality equal to $4$ for the following $8$ values of $N$:

    \begin{center}
\begin{longtable}{|c|c|c|}

\hline
\addtocounter{table}{-1}
$N$ & $g$ & $\beta_{2,2}$\\
    \hline

$320,326,368,658$ & $7$ & $3$\\

\hline

$304,585$ & $8$ & $4$\\

\hline

$528,560$ & $9$ & $5$\\
 
\hline

\caption*{}
\end{longtable}
\end{center}
\end{prop}

\begin{proof}
    For these levels we computed that $\beta_{2,2}=g-4$. Hence there exists a unique $g_4^1$ by Theorem \ref{schreyer}. Since curves $X_0^*(N)$ have a rational cusp, Corollary \ref{unique_g14_cor} now tells us that we have a rational map to $\PP^1$ of degree $4$. Hence all these curves have $\Q$ and $\C$-gonality $4$ as they are not trigonal by \cite{HasegawaShimura2000}.
\end{proof}

The reason why we computed Betti numbers only for curves of genus $g\leq9$ (with the exception of $N=416$) is the Tower theorem.

\begin{thm}[The Tower Theorem]\label{towerthm}
Let $C$ be a curve defined over a perfect field $k$ such that $C(k)\neq\emptyset$ and let $f:C\to \mathbb{P}^1$ be a non-constant morphism over $\overline{k}$ of degree $d$. Then there exists a curve $C'$ defined over $k$ and a non-constant morphism $C\to C'$ defined over $k$ of degree $d'$ dividing $d$ such that the genus of $C'$ is $\leq (\frac{d}{d'}-1)^2$.
\end{thm}

\begin{cor}\cite[Corollary 4.6. (ii)]{NajmanOrlic22}\label{towerthmcor}
    Let $C$ be a curve defined over $\Q$ with $\textup{gon}_\C (C)=4$ and $g(C)\geq10$ and such that $C(\Q)\neq\emptyset$. Then $\textup{gon}_\Q (C)=4$.
\end{cor}

Therefore, for curves $X_0^*(N)$ of genus $g\geq10$ it is enough to prove that $\Q$-gonality is greater than $4$ and the $\C$-gonality will also be greater than $4$.

\section{Morphisms $X_0^*(N)\to Y$}\label{morphism_section}

In this section we construct rational morphisms from curves $X_0^*(N)$ and use them to obtain gonality bounds for the remaining $92$ levels $N$. We first present results that prove the existence of degree $4$ rational morphisms to $\PP^1$.

\begin{prop}\label{pointsearch}
    The curve $X_0^*(N)$ has $\Q$ and $\C$-gonality equal to $4$ for the following $23$ levels \begin{align*}
        N\in\{&218,226,235,237,250,278,339,382,391,393,402,406,\\
        &407,413,418,435,438,465,494,551,574,595,645\}.
    \end{align*}
\end{prop}

\begin{proof}
    All these curves have genus $5$. Using \texttt{Magma} we found a rational divisor that is a sum of $4$ rational points and has Riemann-Roch dimension equal to $2$. We searched for rational points using the \texttt{Magma} function PointSearch().
\end{proof}

\begin{prop}\label{quadpts}
    The curve $X_0^*(N)$ has $\Q$ and $\C$-gonality equal to $4$ for the following $11$ levels
    \[N\in\{288,371,377,410,423,442,663,714,770,798,910\}.\]
\end{prop}

\begin{proof}
    The curve $X_0^*(423)$ has genus $7$ (and $\beta_{2,2}=9$), other curves in the list have genus $5$. For these levels there is no rational divisor of the form $1+1+1+1$ and Riemann-Roch dimension $\geq2$. Therefore, we had to search for higher degree points.

    We searched for quadratic points $(x_0,\ldots,x_k)$ by intersecting the curve $X_0^*(N)$ with hyperplanes of the form
    \[b_0x_0+b_1x_1+b_2x_2=0,\]
    where $b_0,\ldots,b_k\in \Z$ are coprime and chosen up to a certain bound, a similar idea as in \cite[Section 3.2]{Box19}. The main idea is that the first $3$ coordinates of a quadratic points must be linearly dependent over $\Z$. 
    
    We constructed a scheme from the intersection of $X_0^*(N)$ with hyperplanes $b_0x_0+b_1x_1+b_2x_2=0$ and searched for points using the function PointsOverSplittingField(). After that, we input these points and construct degree $4$ divisors out of them, as in Proposition \ref{pointsearch}.  More details are in the code.
     
    In all of these cases we found a function of degree $4$ lying in the Riemann-Roch space of a divisor of the form $1+1+2$, that is, $P_1+P_2+Q+\sigma(Q)$, where $P_1,P_2\in X_0^*(N)(\Q)$, and $Q$ is one of the quadratic points we found.
\end{proof}

We now move on to the morphisms from $X_0^*(N)$ induced by involutions.

\begin{prop}[{\cite[Proposition 1]{Hasegawa1997}, \cite[Proposition 4.15]{bars22biellipticquotients}}]\label{v2prop}
    Let $N$ be an integer divisible by $8$ and let $2^\nu \mid\mid N$. Put $S_\mu=\begin{bmatrix}
        1 & \frac{1}{\mu} \\ 0 & 1
    \end{bmatrix}$. Then $V_2=S_2w_{2^\nu}S_2$ defines an involution on $X_0(N)$ defined over $\Q$ that commutes with all Atkin-Lehner involutions $w_r$.
\end{prop}

Hence, if $8\mid N$, then $V_2$ is an involution of $X_0^*(N)$ defined over $\Q$. Since we can construct degree $2$ morphisms from involutions, we have a degree $2$ rational map $X_0^*(N)\to X_0^*(N)/\left<V_2\right>$.

\begin{prop}[{\cite[Proposition 1]{Hasegawa1997}, \cite[Proposition 4.19]{bars22biellipticquotients}}]\label{v3prop}
    Let $N$ be an integer such that $9 \mid\mid N$. Put $S_\mu=\begin{bmatrix}
        1 & \frac{1}{\mu} \\ 0 & 1
    \end{bmatrix}$. Then $V_3=S_3w_9S_3^2$ defines an involution on $X_0(N)$ and $X_0^*(N)$ defined over $\Q(\sqrt{-3})$. Moreover. if $V_3$ is an involution of $X_0(N)/W$, then it is defined over $\Q$ if and only if $w_9\in W$.
\end{prop}

Hence, if $9\mid\mid N$, then $V_3$ is an involution of $X_0^*(N)$ defined over $\Q$ and have a degree $2$ rational map $X_0^*(N)\to X_0^*(N)/\left<V_3\right>$. A noticeable difference from $V_2$ is that $V_3$ does not necessarily commute with all Atkin-Lehner involutions $w_r$. Thus it does not need to be an involution on all quotients of $X_0(N)$.

We can compute the genera of these quotients by counting the newforms $f$ with $A_f$ in the decomposition of the Jacobian $J_0^*(N)$. We use the following results. Lemma \ref{lift_newform_lemma} describes how to lift newforms to a higher level with desired Atkin-Lehner action and Propositions \ref{v2newform_prop}, \ref{v3newform_prop} tell us how to calculate the genera of quotients $X_0^*(N)/\left<V_2\right>, X_0^*(N)/\left<V_3\right>$.

\begin{lem}[{\cite[Lemma 1]{Hasegawa1997}}]\label{lift_newform_lemma}
    Let $M$ be a positive integer. Let $M_0$ be a positive divisor of $M$ and let $d$ be a positive divisor of $\frac{M}{M_0}$. For a prime divisor $p$ of $M$ define integers $\alpha,\beta,\gamma$ by \[p^\alpha \mid\mid M, \ p^{\alpha-\beta}\mid\mid M', \ p^\gamma\mid\mid d.\] If $f(\tau)\in S_2^0(M')$ is a newform on $\Gamma_0(M')$, then $f$ lifts to an eigenform for all $w_m^M$ with $m\mid\mid M$. In particular, if $f| w_{p^{\alpha-\beta}}^{M'}=\lambda_p'f \ (=\pm f)$ and if $\beta>2\gamma$ (resp. $\beta=2\gamma$), then \[f(d\tau)\pm p^{\beta-2\gamma}\lambda_p'f(p^{\beta-2\gamma}d\tau) \ (\textup{resp. } f(d\tau))\] becomes an eigenform for $w_{p^{\alpha-\beta}}^M$ with eigenvalue $\pm1$ (resp. $\lambda_p'$).
\end{lem}

Now we will consider maps $X_0^*(N)\to X_0^*(N)/\left<V_2\right>, X_0^*(N)/\left<V_3\right>$. For simplicity we will write $f^{(d)}(\tau)=f(d\tau)$ in the following results.

\begin{prop}[{\cite[Proposition 2]{Hasegawa1997}}]\label{v2newform_prop}
    Let $N$ be a positive integer such that $8\mid N$. Let $N'$ be a positive divisor of $N$ and let $d$ be a positive divisor of $\frac{N}{N'}$. Define integers $\alpha,\beta,\gamma$ by \[2^\alpha \mid\mid N, \ 2^{\alpha-\beta}\mid\mid N', \ 2^\gamma\mid\mid d\] so that $N=2^\alpha M$ and $N'=2^{\alpha-\beta}M'$ for odd integers $M'\mid M$- Let $f$ be a newform on $\Gamma_0(N')$ such that $f| w_{2^{\alpha-\beta}}^{N'}=\lambda f$, and put \[g^{(d)}=f^{(d)}\pm 2^{\beta-2\gamma}\lambda f^{(2^{\beta-2\gamma}d)}.\]
    \begin{enumerate}[(i)]
        \item If $\alpha-\beta\geq2$, then \[\begin{cases}
            g^{(d)}| V_2=-g^{(d)} & \textup{if } \beta>\gamma=0,\\
            g^{(d)}| V_2=+g^{(d)} & \textup{if } \beta-\gamma>\gamma>0,\\
            f^{(d)}| V_2=\lambda f^{(d)} & \textup{if } \beta=2\gamma.
            \end{cases}\]
        \item If $\alpha-\beta=1$, then \[\begin{cases}
            (g^{(d)}+\lambda g^{(2d)})| V_2=-(g^{(d)}+\lambda g^{(2d)}) & \textup{if } \beta>\gamma=0,\\
            g^{(d)}| V_2=+g^{(d)} & \textup{if } \beta-\gamma>\gamma>0,\\
            f^{(d)}| V_2=\lambda f^{(d)} & \textup{if } \beta=2\gamma.
            \end{cases}\]
    \end{enumerate}
\end{prop}

\begin{prop}[{\cite[Proposition 3]{Hasegawa1997}}]\label{v3newform_prop}
    Let $N=9M$ be a positive integer such that $3\nmid M$. Let $M'$ be a positive divisor of $M$ and let $d$ be a positive divisor of $\frac{M}{M'}$.
    \begin{enumerate}[(i)]
        \item Let $f$ be a newform on $\Gamma_0(9M)$ such that $f|w_9^{9M'}=+f$. Then $f^{(d)}$ is an eigenform of $V_3$ with eigenvalue $+1$.
        \item Let $f$ be a newform on $\Gamma_0(3M)$ such that $f|w_3^{3M'}=\lambda f$. Then $f^{(d)}+3\lambda f^{(3d)}$ is an eigenform of $V_3$ with eigenvalue $-1$.
    \end{enumerate}
\end{prop}

\begin{ex}[{\cite[Example 2]{Hasegawa1997}}]\label{v2example}
    Let $n=360=8\cdot9\cdot5$. A basis of $S_2^*(360)$ is given by (the newforms are given with their LMFDB labels \cite{lmfdb}) \begin{align*}
        &g_1^{(1)}, \ g_1^{(3)}, \textup{ where } f_1=20.2.a.a,\\
        &g_2^{(1)}, \ g_2^{(2)}, \textup{ where } f_2=30.2.a.a,\\
        &g_3^{(1)}, \textup{ where } f_3=36.2.a.a,\\
        &g_4^{(1)}, \textup{ where } f_4=90.2.a.b,\\
        &g_5^{(1)}, \textup{ where } f_5=120.2.a.a.
    \end{align*}
    We can check that these newforms have the correct Atkin-Lehner signs on \begin{center}
\url{https://www.lmfdb.org/ModularForm/GL2/Q/holomorphic/?level_type=divides&level=360&weight=2&char_order=1}
    \end{center}
    Thus $g(X_0^*(360))=7$. By Proposition \ref{v2newform_prop} we can see that only $g_2^{(2)}$ and $g_5^{(1)}$ form a basis of $S_2^*(360)^{V2}$. Therefore the curve $X_0^*(360)/\left<V_2\right>$ is of genus $2$.
\end{ex}

\begin{ex}\label{v3example}
    Let $n=414=2\cdot9\cdot23$. A basis of $S_2^*(414)$ is given by (the newforms are given with their LMFDB labels \cite{lmfdb}) \begin{align*}
        &g_1^{(1)}, \textup{ where } f_1=69.2.a.a,\\
        &g_2^{(1)}, \textup{ where } f_2=138.2.a.a,\\
        &g_3^{(1)}, \textup{ where } f_3=138.2.a.b,\\
        &g_4^{(1)}, \textup{ where } f_4=207.2.a.b \  (\dim=2).
    \end{align*}
    We can check that these newforms have the correct Atkin-Lehner signs on \begin{center}
\url{https://www.lmfdb.org/ModularForm/GL2/Q/holomorphic/?level_type=divides&level=414&weight=2&char_order=1}
    \end{center}
    
    Thus $g(X_0^*(414))=5$. The newforms of level $3\mid\mid M$ will never lift to eigenforms with eigenvalue $+1$ for $V_3$ because of Proposition \ref{v3newform_prop}. Proposition \ref{v3newform_prop} also tells us that, due to action of $w_9$, the two eigenforms $f_4^{(1)}$ form a basis for $S_2^*(414)^{V2}$. Therefore the curve $X_0^*(414)/\left<V_3\right>$ is of genus $2$.
\end{ex}

\begin{prop}\label{v2levels}
    The curve $X_0^*(N)$ has $\Q$ and $\C$-gonality equal to $4$ for the following $14$ values of $N$. Here $\dim$ denotes the dimensions of newforms in the table.
    \begin{center}
\begin{longtable}{|c|c|c||c|c|c|}

\hline
\addtocounter{table}{-1}
$N$ & newforms for $V_2$ & $\dim$ & $N$ & newforms for $V_2$ & $\dim$\\
    \hline

$192$ & 24.2.a.a, 192.2.a.a & $1,1$ & $208$ & 26.2.a.b, 208.2.a.b & $1,1$\\
\hline
$216$ & 72.2.a.a, 216.2.a.a & $1,1$ & $232$ & 58.2.a.a, 232.2.a.a & $1,1$\\
\hline
$272$ & 34.2.a.a, 272.2.a.a & $1,1$ & $296$ & 37.2.a.a, 296.2.a.a & $1,1$\\
\hline
$328$ & 82.2.a.a, 328.2.a.a & $1,1$ & $336$ & 42.2.a.a, 112.2.a.a & $1,1$\\
\hline
$344$ & 43.2.a.a, 344.2.a.b & $1,1$ & $360$ & 30.2.a.a, 120.2.a.a & $1,1$\\
\hline
$376$ & 376.2.a.b & $2$ & $440$ & 88.2.a.a, 440.2.a.a & $1,1$\\
\hline
$456$ & 57.2.a.a, 152.2.a.a & $1,1$ & $520$ & 65.2.a.a, 520.2.a.a & $1,1$\\
\hline

\caption*{}
\end{longtable}
\end{center}
\end{prop}

\begin{proof}
    For these levels we compute that the genus of $X_0^*(N)/\left<V_2\right>$ is equal to $2$ as in Example \ref{v2example}. The composition map $X_0^*(N)\to X_0^*(N)/\left<V_2\right>\to\PP^1$ is a degree $4$ rational morphism.
\end{proof}

\begin{prop}\label{v3levels}
    The curve $X_0^*(N)$ has $\Q$ and $\C$-gonality equal to $4$ for the following $6$ values of $N$. Here $\dim$ denotes the dimensions of newforms in the table.
    \begin{center}
\begin{longtable}{|c|c|c||c|c|c|}

\hline
\addtocounter{table}{-1}
$N$ & newforms for $V_3$ & $\dim$ & $N$ & newforms for $V_3$ & $\dim$\\
    \hline

$414$ & 207.2.a.b & $2$ & $495$ & 99.2.a.a & $1$\\
\hline
$504$ & 36.2.a.a, 504.2.a.a & $1,1$ & $558$ & 558.2.a.b & $1$\\
\hline
$630$ & 315.2.a.c & $2$ & $990$ & 99.2.a.a, 990.2.a.a & $1,1$\\
\hline

\caption*{}
\end{longtable}
\end{center}
\end{prop}

\begin{proof}
    For these levels we compute that the genus of $X_0^*(N)/\left<V_3\right>$ is equal to $1$ or $2$ as in Example \ref{v3example}. The composition map $X_0^*(N)\to X_0^*(N)/\left<V_3\right>\to\PP^1$ is a degree $4$ rational morphism.
\end{proof}

When $4\mid\mid N$, the curve $X_0^*(N)$ has a modular involution.

\begin{prop}[{\cite[Proposition 7]{Hasegawa1997}}]\label{prop4n}
    Let $N=4M$ with $M$ being odd. We have the isomorphism \[X_0^*(N)\cong X_0(N)/\left<\{w_{p^r}\}_{p\mid M}\cup S_2\right>\cong X_0(2M)/\left<\{w_{p^r}\}_{p\mid M}\right>.\] The first isomorphism is obtained by conjugating $\Gamma_0(N)$ by $S_2w_4S_2=V_2$ and the second by conjugating $\left<\{w_{p^r}\}_{p\mid M}\cup S_2\cup \Gamma_0(N)\right>$ by $\alpha_2=\begin{bmatrix}
        2 & 0\\ 0 & 1
    \end{bmatrix}$. Thus both isomorphisms are defined over $\Q$.
\end{prop}

\begin{prop}\label{degree2maptostar}
    The curve $X_0^*(N)$ has $\Q$ and $\C$-gonality equal to $4$ for the following $18$ levels \[N\in\{212,244,268,292,316,332,340,364,372,396,412,444,460,532,572,620,660,780\}.\]
\end{prop}

\begin{proof}
    All these levels are such that $4\mid\mid N$. Thus we have a degree $2$ rational morphism \[X_0^*(N)\cong X_0(2M)/\left<\{w_{p^r}\}_{p\mid M}\right>\to X_0^*(2M).\] All these curves $X_0^*(2M)$ are of genus $2$ (\cite[Remark 1]{Hasegawa1997}), so we have a degree $4$ rational morphism to $\PP^1$.
\end{proof}

\begin{prop}\label{not_tetragonal_degree2map}
    The curve $X_0^*(N)$ has $\Q$-gonality at least $5$ for the following $20$ levels \begin{align*}
        N\in\{&596,900,948,996,1012,1032,1036,1044,1068,1164,1212,\\
        &1380,1740,1848,1932,2100,2220,2340,2460,2580\}.
    \end{align*}
\end{prop}

\begin{proof}
    All these levels are such that $4\mid\mid N$. Thus we have a degree $2$ rational morphism \[X_0^*(N)\cong X_0(2M)/\left<\{w_{p^r}\}_{p\mid M}\right>\to X_0^*(2M).\] Moreover, we have proven in Propositions \ref{Fp_gonality}, \ref{genus6prop}, and \ref{betti0prop} that these curves $X_0^*(2M)$ are not $\Q$-tetragonal. Therefore, the curves $X_0^*(N)$ are also not $\Q$-tetragonal by Proposition \ref{poonen} (vii).
\end{proof}

\section{Proofs of the main theorems}\label{thmproof_section}

The proof of Theorem \ref{trigonalthm} is given in Proposition \ref{genus4prop} where we determine the fields of definition of trigonal maps for genus $4$ curves (for prime powers $N$ we have $X_0^*(N)=X_0^+(N)$, they were already covered in \cite{Orlic2023}). We now give the proof of Theorem \ref{tetragonalthm}.

\begin{proof}[Proof of Theorem \ref{tetragonalthm}]
    If $N$ is a prime power, then $X_0^*(N)=X_0^+(N)$. All tetragonal curves $X_0^+(N)$ were determined in \cite{Orlic2023}. Let us suppose now that $N$ is not a prime power. Furthermore, all hyperelliptic and trigonal curves $X_0^*(N)$ were determined in \cite{Hasegawa1997} and \cite{HasegawaShimura2000}. Thus we may suppose that the $\C$-gonality of $X_0^*(N)$ is at least $4$.

    The gonality of genus $4$ curves was determined in Proposition \ref{genus4prop}. We gave the list of all possible tetragonal levels $N$ in \Cref{preliminaries_section}. For levels $N$ of genus $g\geq5$ in the table of the theorem we prove the existence of a degree $4$ rational morphism to $\PP^1$ in Propositions \ref{genus6prop}, \ref{betti_positive_prop}, \ref{pointsearch}, \ref{quadpts}, \ref{v2levels}, \ref{v3levels}, and \ref{degree2maptostar}.

    For the remaining levels $N$ we disproved the existence of a degree $4$ map as follows. For genus $6$ curves we used Proposition \ref{genus6prop} to prove that their $\Q$-gonality is larger than $4$. However, the $\C$-gonality is equal to $4$ due to Proposition \ref{poonen} (v).

    For curves of genus $g=7,8,9$ and for $N=416$ ($g=10$) we used Proposition \ref{betti0prop} to prove that the $\C$-gonality is larger than $4$. For other curves of genus $g\geq10$ we used Propositions \ref{Fp_gonality}, \ref{Fp2points}, and \ref{not_tetragonal_degree2map} to prove that their $\Q$-gonality is larger than $4$. By Corollary \ref{towerthmcor}, their $\C$-gonality is then also larger than $4$.

    Proposition \ref{prop_378} treats the only remaining level $N=378$.
\end{proof}

\bibliographystyle{siam}
\bibliography{references}

\end{document}